\documentclass{amsart}
\usepackage{graphicx,url,xy,amssymb}
\usepackage[noadjust]{cite}
\usepackage{xspace}
\newcommand{\RR}{\mathbb R}
\newcommand{\ZZ}{\mathbb Z}
\newcommand{\QQ}{\mathbb Q}
\newcommand{\CC}{\mathbb C}
\newcommand{\z}{\zeta}%
\newcommand{\OO}{\mathcal O}
\newcommand{\cC}{\mathcal{C}}
\newcommand{\bcC}{\bar{\cC}}

\newcommand{\group}[1]{\mathbf{#1}}
\DeclareMathOperator{\SL}{SL}
\DeclareMathOperator{\GL}{GL}
\newcommand{\G}{\Gamma}
\newcommand{\bG}{\group{G}}

\newcommand{\V}{\mathcal{V}}
\DeclareMathOperator{\Tr}{Tr}
\DeclareMathOperator{\Res}{Res}
\newcommand{\fh}{\mathfrak{h}}
\newcommand{\ip}[1]{\left \langle #1 \right \rangle}
\newcommand{\Vor}{Vorono\"{\i}\xspace}
\newcommand{\pphi}{\hat{\phi}} 

\newcommand{\vect}[1]{\begin{bmatrix} #1 \end{bmatrix}}
\newcommand{\mat}[1]{\begin{bmatrix} #1 \end{bmatrix}}%

\theoremstyle{plain}
\newtheorem{thm}{Theorem}[section]
\newtheorem{lem}[thm]{Lemma}
\newtheorem{prop}[thm]{Proposition}

\theoremstyle{definition}
\newtheorem{defn}[thm]{Definition}
\newtheorem*{remark}{Remark}

\begin{document}


\title{Binary Hermitian forms over a cyclotomic field}
\author{Dan Yasaki}
\address{Department of Mathematics and Statistics\\146 Petty Building\\ University of North Carolina at Greensboro\\Greensboro, NC 27402-6170}
\email{d\_yasaki@uncg.edu}
\date{}
\thanks{The original manuscript was prepared with the \AmS-\LaTeX\ macro
system.}

\begin{abstract}
Let $\z$ be a primitive fifth root of unity and let $F$ be the cyclotomic field $F=\QQ(\z)$.  Let $\OO\subset F$ be the ring of integers.  We compute the \Vor polyhedron of binary Hermitian forms over $F$ and classify $\GL_2(\OO)$-conjugacy classes of perfect forms.  The combinatorial data of this polyhedron can be used to compute the cohomology of the arithmetic group $\GL_2(\OO)$ and Hecke eigenforms.
\end{abstract}
\maketitle

\bibliographystyle{../amsplain_initials} 
\begin{section}{Introduction}
Let $F/\QQ$ be a number field, and let $\OO$ denote its ring of integers.  There exists an algorithm to compute all of the $\GL_n(\OO)$-equivalence classes of perfect $n$-ary quadratic forms over $F$ once an initial perfect form is found \cite{Gmod, koecher}. This is investigated in the totally real number field case in \cite{perfect,leiants}.   We remark that a different notion of perfection has been investigated in \cite{BCI, Cou,Pohst}. In this paper, we consider the cyclotomic field $\QQ(\z_5)$. 

Although Hermitian forms over number fields is of interest in its own right, our main motivation for this computation comes from an investigation of a Taniyama-Shimura type correspondence over $F$, relating Hecke eigenforms over $F$ with integer eigenvalues and elliptic curves over $F$.   Due to the work of Wiles et al.~ such a correspondence is known to exist for $F=\QQ$ \cite{darmon.ts}.  It is an open problem for $[F:\QQ] >1$.  This has been investigated for $F$ an imaginary quadratic field in \cite{Whthesis,Bythesis,Lithesis,Cr,CrWh,sengun.quad} and for $F$ a real quadratic field in \cite{De.rt5, De.ell}.  In an ongoing project joint  with P.~Gunnells, F.~ Hajir, and ~D.~Ramakrishnan, we are investigating the complex quartic field $F =\QQ(\z_5)$. To this end, we first compute the \Vor polyhedron associated to binary Hermitian forms over $\QQ(\z_5)$.  The \Vor polyhedron provides a combinatorial structure in which to perform the Hecke eigenvalue computations.  

The results of the computations in this paper could also be used to compute the elliptic points of $\GL_2(\OO)$ using techniques of \cite{Yasfix}.  Since the forms and their minimal vectors are given explicitly here, one could also extract invariants such as an additive analogue of the Hermite constant for $\QQ(\z_5)$ \cite{BCI, Pohst}.

The paper is organized as follows.  In Section~\ref{sec:notation}, we set notation.  In Sections~\ref{sec:sahc}--\ref{sec:primitive} we recall the \Vor polyhedron and its relation to Hermitian forms over $\QQ(\z_5)$.  The \Vor polyhedron for $\QQ(\z_5)$ is computed in Section~\ref{sec:voronoi}.

I would like to thank P.~Gunnells for many helpful tips, tricks and explanations.  I would also like to thank F.~Hajir for helpful conversations.
\end{section}
\begin{section}{Notation}\label{sec:notation}
First we set notation and recall and collect a few basis facts from algebraic number theory that will be used later. Next we set some notation for the algebraic group associated symmetric space we will be considering.
\begin{subsection}{Field}
Let $\z=\z_5 = e^{2\pi i/5}$ be a primitive fifth root of unity, and let $F$ be the cyclotomic field $F=\QQ(\z)$.  Let $\OO \subset F$ denote the ring of integers, and let $\bar{\cdot}$ denote complex conjugation.  Let $k \subset F$ denote the real subfield $k=\QQ(\sqrt{5})$, and let $\OO_k$ denote its ring of integers.  Then $u_5=(1+\sqrt{5})/2$ is a fundamental unit for $k$.

Let $\iota=(\iota_1,\iota_2)$ denote the (non complex conjugate) embeddings 
\[\iota : F \to \CC \times \CC\]
given by sending $\sqrt{5}$ to $(\sqrt{5},-\sqrt{5})$, or equivalently given by sending $\z$ to $(\z,\z^3)$.
Denote the non-trivial embedding by $\cdot'$.  Specifically, for $\alpha \in F$, let $(\alpha, \alpha')$ denote $\iota(\alpha)$. 
\end{subsection}

\begin{subsection}{Symmetric space}
Let $\bG$ be the $\QQ$-group $\Res_{F/\QQ}(\GL_2)$ and let $G=\bG(\RR)$ the
corresponding group of real points.  Let $K\subset G$ be a maximal
compact subgroup, and let $A_G$ be the identity component of the
maximal $\QQ$-split torus in the center of $G$.  Then the symmetric
space associated to $G$ is $X=G/KA_G$.  Let $\Gamma \subseteq
\GL_2(\OO)$ be a finite index subgroup. 
\end{subsection}

\begin{subsection}{Binary Hermitian forms over $F$}
\begin{defn}
A \emph{binary Hermitian form over $F$} is a map $\phi:F^2 \to k$ of the form 
\[\phi(x,y) = a x \bar{x} + b x \bar{y} + \bar{b}\bar{x} y + c y \bar{y},\]
where $a,c \in k$ and $b\in F$.
\end{defn}
Note that $\pphi = \phi + \phi'$ takes values in $\QQ$.  Indeed, $\pphi$ is precisely the composition $\Tr_{k/\QQ} \circ \phi$, and by choosing a $\QQ$-basis for $F$, $\pphi$ can be viewed as a quaternary quadratic form over $\QQ$.  In particular, it follows that $\pphi(\OO^2)$ is discrete in $\QQ$.  Using $\pphi$, we can define minimal vectors and perfection.

\begin{defn}  
The \emph{minimum of $\phi$} is \[m(\phi)=\inf_{v \in \OO^2 \setminus \{ 0\}} \pphi(v).\]
A vector $v\in \OO^2$ is \emph{minimal vector} for $\phi$ if $\phi(v)=m(\phi)$.  The set of minimal vectors for $\phi$ is denoted $M(\phi)$. 
\end{defn}

\begin{defn}
A Hermitian form over $F$ is \emph{perfect} if it is uniquely determined by $M(\phi)$ and $m(\phi)$. 
\end{defn}
Perfection is a $\QQ$-homothety invariant.  Indeed, for $c \in \QQ$, $M(c \phi) = M(\phi)$ and $m(c\phi) = c m(\phi)$.
\end{subsection}
\end{section}
\begin{section}{Self-adjoint homogeneous cone}\label{sec:sahc}
\begin{subsection}{Cone of Hermitian forms over $\CC$}
Every Hermitian form over $\CC$ can be represented by a Hermitian matrix.  Let $C$ be the cone of positive definite binary complex Hermitian forms, viewed as a subset of $V$, the $\RR$-vector space of $2\times 2$ complex Hermitian matrices.  Let $\cdot ^*$ denote complex conjugate transpose.  Then the usual action of $\GL_2(\CC)$ on $C$ given by 
\begin{equation}\label{eq:action}
(g \cdot\phi)(v) = \phi(g^* v), \quad \text{where $g \in \GL_2(\CC)$ and $\phi\in C$.}\end{equation}
Equivalently, if $A$ is the Hermitian matrix representing $\phi$, then $g \cdot \phi = g Ag^*$.  In particular a coset $gU(2) \in \GL_2(\CC)/U(2)$ can be viewed as the Hermitian form associated to the matrix $gg^*$.
\end{subsection}

\begin{subsection}{Cone of Hermitian forms over $F$}\label{sec:binaryform}
Let $\V = V \times V$, and let $\cC$ be the cone $\cC = C \times C \subset \V$.  Since $C$ is the space of positive-definite Hermitian forms on $\CC^2$, we can use $\iota$ to view $\cC$ as the space of forms on $F^2$.  Specifically, for $\phi=(\phi_1,\phi_2) \in \cC$ and $v \in F^2$, we define 
\[\phi(v)=\phi_1(\iota_1(v))+\phi_2(\iota_2(v)).\] 
\end{subsection}

\begin{subsection}{$\cC$ as a symmetric space}
The embedding $\iota$ gives an isomorphism
\begin{equation}\label{eq:Gmap}
G \to \GL_2(\CC) \times \GL_2(\CC).\end{equation}
Under this identification, $\iota(A_G) = \{(rI,rI)\;|\; r > 0\}$, where $I$ is the $2\times 2$ identity matrix.

Combining \eqref{eq:action} and \eqref{eq:Gmap}, we get an action of $G$ on $\cC$.  Let $\phi_0$ denote the binary Hermitian form represented by $I$.  Then the stabilizer on $\G$ of $\iota(\phi_0)$ is a maximal compact subgroup $K$.  The group $A_G$ acts on $\cC$ by positive real homotheties, and we have 
\[X=\cC/\RR_{>0} \simeq X_\CC \times X_\CC \times \RR, \]
where $X_\CC$ is the symmetric space for $\SL_2(\CC)$.
\end{subsection}

\begin{subsection}{$\cC$ as a self-adjoint homogeneous cone}
The cone $\cC$ is a \emph{self-adjoint homogeneous cone} \cite{A}.  Specifically, 
\begin{enumerate}
\item There is a scalar product $\ip{\cdot, \cdot}$ on $\V$ such that 
\[\cC=\left\{\phi \in \V \;|\; \ip{\phi,\psi} >0 \quad \text{for all $\psi \in \bcC \setminus \{0\}$.}\right\}\]
\item  The group $\{g \in \GL(V):g \cdot \cC = \cC\}^0$ acts transitively on $\cC$.
\end{enumerate}

Once the basepoint is fixed, we pick it to be $\iota(\phi_0)$, there is a canonical inner product making $\cC$ self-adjoint.  By viewing $\V$ as a subspace of $4\times 4$ block diagonal complex matrices, there is a natural multiplication on $\V$.  Then for $\phi,\psi \in \V$, the map $\ip{\phi,\psi}=\Tr(\phi \psi)$ defines a scalar product.  

Let $\bar{\cC}$ denote the closure of $\cC$ in $\V$.  Each vector $w \in \CC^2$ gives a rank 1 Hermitian form $w w^*$ (here $w$ is viewed as a column vector).  Combined with $\iota$, we get a map $q\colon F^2 \to \bar{\cC}$ given by 
\[q(v)=(v v^*,v' v'^*).\]

The interpretation of $\cC$ as a space of forms over $F$ is reflected in the scalar product.  Specifically, suppose $\phi \in \V(\QQ)$ and $v \in \OO^2$.  Then 
\[\ip{\phi,q(v)}=\phi(v)=\Tr_{k/\QQ}(v^*A_\phi v).\]
\end{subsection}
\end{section}
\begin{section}{\Vor polyhedron}\label{sec:vorpoly}
\begin{subsection}{Rational structure of $\V$}
Fix $\Lambda \subset \V$ to be the lattice generated by $q(v)$ for $v\in \OO^2$.  This choice of $\Lambda$ allows us to talk about rational points of $\V$.  We will refer to points of $\cC \cap \V(\QQ)$ as \emph{Hermitian forms over $F$}.

\begin{defn}
Let $\phi$ be a Hermitian form over $F$.  A $2 \times 2$ matrix $A$ with coefficients in $F$ is \emph{associated to $\phi$} if $\phi=\iota(A)$.  
\end{defn}
Since the map $\iota$ is injective, this matrix is unique and will be denoted $A_\phi$.  For $v\in F^2$, then $\phi(v)$ is just the trace.  Specifically, if $\phi \in \V(\QQ)$ with associated matrix $A_\phi$, then 
\[\phi(v)=\Tr_{k/\QQ}(v^*A_\phi v).\]   
Let $R(v)$ be the ray $\RR_{>0} \cdot q(v) \subset \bcC$.  Note that 
\[q(cv) = (\iota_1(c)^2q_1(v),\iota_2(c)^2q_2(v))\] 
so that if $c \in \QQ$, then $q(cv)\in R(v)$, and in particular
$q(-v)=q(v)$.  The set of \emph{rational boundary components}
$\cC_{1}$ of $\cC$ is the set of rays of the form $R (v)$, $v\in
F^{2}$ \cite{A}.  

\begin{defn}
The \emph{\Vor polyhedron $\Pi$} is the closed convex hull in $\bcC$
of the points $\cC_{1}\cap \Lambda \setminus \{0\}$.
\end{defn}
\end{subsection}
\begin{subsection}{\Vor decomposition}
By construction $\GL_2(\OO)$ acts on $\Pi$.  By taking the cones on
the faces of $\Pi$, one obtains a \emph{$\Gamma$-admissible
decomposition} of $\cC$ for $\Gamma =\GL_{2} (\OO)$ \cite{A}.
Essentially this means that the cones form a fan in $\bcC$ and that
there are finitely many cones modulo the action of $\GL_{2} (\OO)$.
Since the action of $\GL_{2} (\OO)$ commutes with the homotheties,
this decomposition descends to a $\GL_{2} (\OO)$-equivariant
tessellation of $X$.\footnote{If one applies this construction to
$F=\QQ$, one obtains the Farey tessellation of $\fh $, with tiles
given by the $\SL_{2} (\ZZ)$-orbit of the ideal geodesic triangle with
vertices at $0,1,\infty$.  }

We call this decomposition the \emph{\Vor decomposition}.  We call the
cones defined by the faces of $\Pi$ \emph{\Vor cones}, and we refer
to the cones corresponding to the facets of $\Pi$ as \emph{top cones}.
The sets $\sigma \cap \cC$, as $\sigma$ ranges over all top cones,
cover $\cC$.  Given a point $\phi \in \cC$, there is a finite
algorithm that computes which \Vor cone contains $\phi$ \cite{Gmod}.

For some explicit examples of the \Vor decomposition over real
quadratic fields, we refer to \cite{Ong,ants}.
\end{subsection}
\end{section}

\begin{section}{Primitivity and Minimal vectors}\label{sec:primitive}
There is another notion of minimal vectors that is explored in \cite{A,Gmod}.  In this section, we show that our notion of minimality agrees with theirs in this case by first defining primitive points for the case of interest and examining both notions of minimal vectors.  This should follow from general results of \cite{koecher}.
\begin{subsection}{Primitive points}  There are several notions of primitivity that we will need.
\begin{defn}
A vector $v=\vect{\alpha\\\beta} \in \OO^2$ is \emph{primitive} if the ideal $(\alpha,\beta)=\OO$.
\end{defn}

\begin{prop}\label{prop:torsion}
Let $u,v \in \OO^2$ be primitive vectors. Then $q(u)=q(v)$ if and only if $u=\tau v$ for some torsion unit $\tau \in \OO$.
\end{prop}
\begin{proof}
One direction is immediate.  Specifically, if $\tau \in \OO$ is a torsion unit, then $\tau \bar{\tau}=\tau' \bar{\tau}'=1$, and thus $q(\tau v)=q(v)$ follows from the definition of $L$.

For the converse, if $v=\vect{\alpha\\ \beta}$ then $u=\vect{\xi \alpha\\\eta \beta}$ for some $\xi,\eta \in F$.  Since $q(u)=q(v)$, 
\[\xi \bar{\xi}=\eta \bar{\eta}=1 \quad \text{and}\quad \xi\bar{\eta}=\bar{\xi}\eta=1.\]  It follows that $u=\xi v$ for some $\xi \in F$ that satisfies $\xi\bar{\xi}=1$.  Write $\xi$ as $\xi=\lambda/\mu$ for some $\lambda, \mu \in \OO$ with $(\lambda, \mu)=\OO$.  Since $\xi \alpha \in \OO$, it follows that $\mu \mid \alpha$.   Similarly, since $\xi \beta \in \OO$, it follows that $\mu \mid \beta$.  Since $v$ is primitive, $\mu$ is a unit, and hence $\xi \in \OO$. Then $\xi$ must be a torsion unit because $\xi \bar{\xi}=1$.
\end{proof}

\begin{defn}
A form $\phi \in \Lambda \cap \bar{\cC}$ is \emph{primitive form} if there exists a matrix $A=\mat{a&c\\\bar{c}&b}$ with $\gcd(a,b)=1$ and $\iota(A)=\phi$. 
\end{defn}
Note that if $\phi=\iota(A)$ is primitive, then $c \in \OO$ and $a,b\in \OO_k$ are totally positive.

\begin{lem}\label{lem:norms}
Let $a,b \in \OO_k^\times$ be totally positive with $\gcd(a,b)=1$.  Then $ab \in N_{F/k}(F)$ if and only if $a \in N_{F/k}(F)$ and $b \in N_{F/k}(F)$.
\end{lem}
\begin{proof}
If $a$ and $b$ are norms, then their product is clearly a norm.  For the other direction, suppose $m=ab=N_{F/k}(\xi)$ for some $\xi \in F$.  Since $\gcd(a,b)=1$, it suffices to consider the case where $m$ is square-free in $k$.  Factor $(\xi)$ as a product of prime ideals $(\xi)=\prod \pi$ in $F$.  Since $m$ is square-free, we must have that each $\pi$ lies over $5$ or a rational prime $p$ such that $p \equiv 1 \bmod{5}$.  It follows that $(\xi)$ has a factorization in $k$
\[(\xi)=\rho \wp_1 \wp_2 \cdots \wp_n,\] where the $\wp_i$ are prime ideals in $k$ and $\rho$ is either trivial or generated by $\sqrt{5}$.  Each of these factors has a generator which is a norm, and so it follows that $(a)$ and $(b)$ have generators which are norms.  Specifically, $e_1a$ and $e_2b$ are norms for some choice of units $e_1,e_2 \in \OO^*_k$.  The fundamental unit $u_5$ of $k$ is positive in one embedding into $\RR$ and negative in the other.  In particular since $a$, $b$, $e_1a$, and $e_2b$ are totally positive, $e_1$ and $e_2$ must be even powers of $u_5$.  It follows that $a$ and $b$ are themselves norms.
\end{proof}

\begin{prop}\label{prop:primitive}
Let $\phi \in \Lambda \cap \bar{\cC}$ be a primitive rank $1$ form.  Then there exists a primitive vector $v\in \OO^2$ such that $q(v)=\phi$. 
\end{prop}
\begin{proof}
Since $\phi$ is primitive, there exists a matrix $A=\mat{a&c\\\bar{c}&b}$ with $\gcd(a,b)=1$ and $\iota(A)=\phi$.  Since $\phi$ is rank $1$, $c\bar{c}=N_{F/k}(c)=ab$.  By Lemma~\ref{lem:norms}, there exist $\alpha,\beta \in F$ with $N_{F/k}(\alpha)=a$ and $N_{F/k}(\beta)=b$.  Since $\gcd(a,b)=1$, it follows that $v=(\alpha,\beta)^t$ is a primitive vector with $q(v)=\phi$ as desired.
\end{proof}

\begin{remark}
Since $\Lambda$ is a lattice, there is another notion of primitivity for $\phi \in \Lambda$.  Specifically, $\phi$ is \emph{primitive in $\Lambda$} if $\phi$ can be extended to a basis of the lattice.  Note that Proposition~\ref{prop:primitive} shows that primitive rank $1$ forms correspond to primitive vectors in $\OO^2$ and primitive vectors in $\OO^2$ give rise to primitive rank $1$ forms.  However, while primitive rank $1$ forms are primitive in $\Lambda$, there are many vectors in primitive in $\Lambda$ that are not primitive as forms.

Even when we restrict ourselves to considering vectors that are primitive in $\Lambda$ and rank $1$, if we relax the condition of $\gcd(a,b)=1$, we can no longer guarantee that this vector comes from a primitive vector in $\OO^2$.  For example, $4+u_5$ is totally positive, but there does not exist $\alpha \in F$ such that $\alpha \bar{\alpha}=4+u_5$.  Thus the point $\iota\left(\mat{4+u_5&0\\0&0}\right)$ is not in $q(\OO^2)$.  However, we do have the following result, which basically says that although there are rank 1 rational forms which are not in the image of $\OO^2$, they are contained in $\OO_k \cdot q(\OO^2)$.     
\end{remark}
\begin{prop}\label{prop:rank1}
Let $\phi \in \Lambda' \cap \cC_1$ be a rank $1$ form.  Then there exists $\alpha \in \OO_k$  totally positive and a primitive vector $v\in \OO^2$ such that $\alpha \cdot q(v)=\phi$. 
\end{prop}
\begin{proof}
Write $\phi$ as $\phi=\iota(A)$, where $A=\mat{a&c\\\bar{c}&b}$ for some $a,b \in \OO_k$ totally positive and $c \in \OO$.  If $\phi$ is primitive, the result follows from Proposition~\ref{prop:primitive} with $\alpha =1$. If $\phi$ is not primitive, there exists an $\alpha \in \OO_k$ totally positive such that $\alpha \mid a$ and $\alpha \mid b$.  Since $\phi$ is rank $1$, it follows that $c \bar{c}=ab$.  In particular $\alpha^2 \mid c\bar{c}$.  Since $\alpha \in \OO_k$, we have that $\alpha \mid c$ and $\alpha \mid \bar{c}$.  Thus we have $\phi=\iota(\alpha A_0 )$, where $A_0 = \alpha^{-1}A$.  Since $A_0$ corresponds to a primitive rank $1$ form, Proposition~\ref{prop:primitive} implies that there exists a primitive vector $v \in \OO^2$ such that $q(v)=\iota(A_0)$.  The result follows.
\end{proof}

\begin{prop}\label{prop:positive}
Let $b \in \OO_k^+$.  Then there exists an $\alpha \in \OO\setminus \{0\}$ and $t \in \OO_k^+ \cup \{0\}$ such that 
\[b=\alpha \bar{\alpha} +t.\]
\end{prop}
\begin{proof}
The square of the fundamental unit $\eta = u_5^2 \in \OO_k$ is totally positive.  Since $\eta=u_5 \bar{u}_5$, multiplication by $\eta$ acts on the $\OO^+$ and preserves $N_{F/k}(\OO)$.  In particular it suffices to show the result for a fundamental domain for the action of $\eta$ on $\OO^+$.  Once can take the positive cone spanned by $1$ and $\eta^2$.  It is clear that every point in the cone has the form $\eta^2+t$ for some $t\in \OO_k^+$ except for $1$ and $2$.  The condition is trivially satisfied for $1$ and $2$.  
\end{proof}
\end{subsection}
\begin{subsection}{Minimal vectors}
There is another notion of minimal vectors and perfect forms described in \cite{Gmod}.  Specifically one can define 
\[\hat{m}(\phi) = \inf_{\psi \in \Lambda' \cap \cC_1 } \ip{\phi,\psi}\] and 
\[\hat{M}(\phi) = \{\psi \in \Lambda' \cap \cC_1 \;|\;\ip{\phi, \psi} = 1\}.\]
Notice that if $\psi \in \hat{M}(\phi)$, then $\psi$ is primitive in $\Lambda$. It is clear that  $\hat{m}(\phi) \leq m(\phi)$.  If $\hat{m}(\phi) = m(\phi)$, then 
\[\{q(v)\;|\; v \in M(\phi)\} \subseteq \hat{M}(\phi).\]

\begin{prop}\label{prop:mhat}
Let $\phi \in \cC$.  Then 
\[\hat{m}(\phi)=\min_{\substack{b \in \OO_k\\b \gg 0}} m(b \cdot \phi)\]
\end{prop}
\begin{proof}
 For $\psi \in \Lambda' \cap \cC_1$, there exists $b \in \OO_k$, $b \gg0 $ and $v \in \OO^2$ such that $\psi = b \cdot q(v)$ by Proposition~\ref{prop:rank1}.  It follows that 
\begin{align*}
\ip{\phi,\psi}=\ip{\phi,b\cdot q(v)}=\Tr_{k/\QQ}(A_\phi b v v^*)=\Tr_{k/\QQ}(v^*bA_\phi v).
\end{align*}
Then 
\begin{align*}
\hat{m}(\phi)=\inf_{\psi \in \Lambda' \cap \cC_1 } \ip{\phi,\psi}=\inf_{\substack{v \in \OO^2\\b \in \OO_k,b \gg0} } \Tr_{k/\QQ}(v^*bA_\phi v)=\inf_{\substack{b \in \OO_k,b \gg0}} m(b \cdot \phi).
\end{align*}
\end{proof}

\begin{prop}\label{prop:msame}
Let $\phi  \in \cC$.  Then $ m(\phi) = \hat{m}(\phi)$.
\end{prop}
\begin{proof}
By Proposition~~\ref{prop:mhat}, 
\[\hat{m}(\phi) = \inf_{b \in \OO_k, b\gg0}m(b \cdot \phi).\]
In particular $\hat{m}(\phi) \leq m(\phi) $, and to prove the result it suffices to show that 
\begin{equation}
m(\phi) \leq m(b \cdot \phi) \quad \text{for every $b \in \OO_k, b\gg0$.}
\end{equation}
By Proposition~\ref{prop:positive}, there exists an $\alpha \in \OO$ and $t\in \OO_k, t\gg0$ such that $b=\alpha \bar{\alpha} +t$.  We compute 
\begin{align*}
\Tr_{k/\QQ}(v^* bA_\phi v)&= \Tr_{k/\QQ}(v^* (\alpha \bar{\alpha}+t)A_\phi v)\\
&=\Tr_{k/\QQ}((\alpha v)^* A_\phi (\alpha v))+\Tr_{k/\QQ}(tv^*A_\phi v).
\end{align*}
We have that $\Tr_{k/\QQ}((\alpha v)^* A_\phi (\alpha v)) \geq m(\phi)$.  Furthermore $\Tr_{k/\QQ}(tv^*A_\phi v) >0$ since $t$ and $v^*A_\phi v$ are both totally positive.  The result follows.
\end{proof}
\end{subsection}
\end{section}

\begin{section}{\Vor cones for $F$}\label{sec:voronoi}
In this section, we describe the $\GL_2(\OO)$ conjugacy classes of \Vor cones.  We implement the method described in \cite{Gmod}.
\subsection{Perfect forms}
We find one perfect form $\phi$, with associated matrix  
\[A_\phi=\frac{1}{5}\mat{
\z^3 + \z^2 + 3 & \z^3 - \z^2 + \z - 1 \\
-2\z^3 - \z - 2 & \z^3 + \z^2 + 3
}\]
This is done using \verb+MAGMA+ \cite{magma} as follows.
\begin{enumerate}
\item Fix a list $L \subset \OO^2$ of vectors large enough so that the conditions \[\{\phi(v)=1\;|\;v \in L\}\] has a unique solution. \label{it:1}
\item Ensure that $\phi$ is positive definite.\label{it:2}
\item Check that $L\subseteq M(\phi)$.\label{it:3}
\end{enumerate}
Step \eqref{it:1} is accomplished by including more vectors into the list $L$ until the linear system has a unique solution. Steps \eqref{it:2} and \eqref{it:3} are accomplished by picking a $\ZZ$-basis for $\OO^2$ and expressing $\phi$ as a quadratic form on $\ZZ^8$.   

The perfect form $\phi$ has $240$ minimal vectors.  It is clear that if $v \in M(\phi)$ then $\tau v \in M(\phi)$ for any torsion unit $\tau \in \OO$.  There are $24$ minimal vectors (modulo torsion units).  
Let $\omega$ denote the unit $\omega = \z+\z^2$.  Then the form $\phi$ has (modulo torsion units) minimal vectors
\begin{multline}\label{eq:perfectmin} 
\vect{
            -\z + 1 \\ \z^3 + 1
            },
            \vect{
            -\z^3+1 \\ 1
            },
            \vect{
            1 \\ -\omega
            },
            \vect{
            1 \\ -\z^2
            },
            \vect{
            1 \\ 0
            },
            \vect{
            1 \\ \z^3
            },
            \vect{
            1 \\ -\z^2 + 1
            },
            \vect{
            1 \\ 1
            },
            \vect{
            1 \\ \z^3 + 1
            },\\
            \vect{
            1 \\ \z + 1
            },
            \vect{
            1 \\ \z^3 + \z + 1
            },
            \vect{
            1 \\ -\z^4
            },
            \vect{
            \omega^{-1} \\ \z^4
            },
            \vect{
            \omega^{-1} \\ \z^4-1
            },
            \vect{
            \omega^{-1} \\ -1
            },
            \vect{
            \omega^{-1} \\ -\z^3 - 1
            },
\\
            \vect{
            \omega^{-1} \\ -\z^3 - \z^2 - 1
            },
            \vect{
            \omega \\ \omega + 1
            },
            \vect{
            \omega \\ -\z^3
            },
            \vect{
            \omega \\ 0
            },
            \vect{
            \omega \\ \z^2
            },
            \vect{
            \omega \\ \omega
            },
            \vect{
            0 \\ 1
            },
            \vect{
            0 \\ \omega
            }
. 
\end{multline}%
By Proposition~\ref{prop:torsion}, these give rise to $24$ distinct points in $\bcC$.  These $24$ points are vertices of a top-dimensional \Vor cone.  We show in \S~\ref{sec:topcone} that this is the only top-dimensional \Vor cone modulo $\GL_2(\OO)$. 
\begin{prop}
There is $1$ $\GL_2(\OO)$-conjugacy class of $8$-dimensional cone.  The corresponding perfect form $\phi$ has (modulo torsion units) $24$ minimal vectors.
\end{prop}
\subsection{Top cones}\label{sec:topcone}
The perfect forms correspond to 8-dimensional \Vor cones.  Specifically, each top cone corresponds to a facet $F$ of the \Vor polyhedron.  There is a unique point $\phi_F \in \cC \cap \V(\QQ)$ \cite{Gmod} such that \begin{enumerate}
\item $F=\{x \in \Pi \; |\; \ip{x,\phi_F}=1\}$, and
\item for all $x \in \Pi \setminus F$, we have $\ip{x, \phi_F} >1$.
\end{enumerate}
Let $Z_F$ be the \emph{minimal vertices of $\phi_F$}, the finite set of points $x \in \Lambda' \cap \cC_1$ such that $\ip{x, \phi_F}=1$.  Then $F$ is the convex hull of $Z_F$.  The form $\phi_F$ is a perfect form, and 
\[\{q(v):v \in M(\phi_F)\} = Z_F.\]

For $\gamma \in \GL_2(\OO)$, let $\Theta \gamma = (\gamma^*)^{-1}$.  The action of $\GL_2(\OO)$ on perfect forms, minimal vectors, and minimal vertices are related by the following.
\begin{prop}
Let $F$ be a top cone and let $\phi_F$ be the corresponding perfect form with minimal vectors $M(\phi_F)$ and minimal vertices $Z_F$.  For $\gamma \in \GL_2(\OO)$, 
\begin{enumerate}
\item $\Theta \gamma \cdot \phi = \phi_{\gamma \cdot F}$, and 
\item $M(\Theta \gamma \cdot \phi) = \gamma \cdot M(\phi)$.
\end{enumerate}
In particular, two perfect forms are $\GL_2(\OO)$-conjugate if and only if their minimal vectors or minimal vertices are $\GL_2(\OO)$-conjugate.
\end{prop}

Thus to classify the perfect forms modulo $\GL_2(\OO)$, one can instead classify the top-dimensional \Vor cones.  The top-dimensional cone corresponding to $\phi$, denoted $\cC_\phi$ is has faces given by the convex hull of $\{q(v)\;|\;v \in M(\phi)\}$.  The program \verb+Polymake+ \cite{polymake} is used to compute the convex hull.  There are $118$ codimension $1$ faces, corresponding to $118$ neighboring top dimensional \Vor cones and $118$ perfect forms.  There are $14$ faces with $12$ vertices, $80$ faces with $9$ vertices, and $24$ faces with $7$ vertices. Using \verb+MAGMA+, the stabilizer $S_\phi$ of $\phi$ is computed.  The group $S_\phi$ has order $600$ and \verb+MAGMA+ type $\ip{600,54}$.

In order to cut down the number of computations that need to be made, we first classify the faces of $\cC_\phi$ modulo $S_\phi$, the stabilizer of $M(\phi)$.  Indeed, let $\psi$ be a perfect form such that $\cC_\psi$ meets $\cC_\phi$ in a codimension 1 face $F$.  If $\gamma \in S_\phi$, then $\gamma \cdot \psi$ is the perfect form corresponding to top cone $\gamma \cdot \cC_\psi$, which meets $\cC_\phi$ in a codimension $1$ face $\gamma \cdot F$.
 
It is clear that $S_\phi$ conjugate faces must have the same number of vertices. One computes that there are $3$ orbits of faces with $12$ vertices.  One orbit $S_\phi \cdot F_1$ consists of $12$ faces, and the other $2$ orbits consist of $1$ element each, denoted $F_2$ and $F_3$.  One shows that $F_2$ is $\GL_2(\OO)$ conjugate to $F_3$ by computing a matrix that sends $F_2$ to $F_3$.  There are $4$ orbits of faces with $9$ vertices.  The orbits $S_\phi \cdot F_4$, $S_\phi \cdot F_5$, $S_\phi \cdot F_6$, and $S_\phi \cdot F_7$ consists of $20$ faces each. There are $2$ orbits of faces with $7$ vertices.  The orbits $S_\phi \cdot F_8$ and $S_\phi \cdot F_9$ consist of $12$ faces each.

To show that there is only $1$ $\GL_2(\OO)$ class of perfect form, it suffices to show that the perfect forms $\phi_i$ associated to the top cones that neighbor $\cC_\phi$ at face $F_i$ are each $\GL_2(\OO)$ conjugate to $\phi$.  To this end, we use the following lemma.

\begin{lem}\label{lem:face}
Let $F,F'$ be two faces of $\cC_\phi$ with associated perfect forms $\psi,\psi'$.  If $\gamma \in \GL_2(\OO) \setminus S_\phi$ and $\gamma \cdot F=F'$, then both $\psi $ and $\psi'$ are $\GL_2(\OO)$-conjugate to $\phi$.
\end{lem}
\begin{proof}
Since $\cC_\phi$ is a top-dimensional cone, and $F$ and $F'$ are codimension $1$ faces of $\cC_\phi$, we must have that $\gamma \cdot \cC_\phi=\cC_\phi$ or   $\gamma \cdot \cC_\phi=\cC_{\psi'}$.  Since $\gamma \not \in S_\phi$, it follows that $\gamma \cdot \cC_\phi=\cC_{\psi'}$.  Thus $\Theta \gamma \cdot \phi = \psi'$.  Repeating the argument using $\gamma^{-1}$ shows that $\Theta \gamma^{-1} \cdot \phi = \psi$.
\end{proof}

\subsection{Lower cones}
In an analogous way, one classifies the lower dimensional faces.  
\begin{thm}\label{thm:coneclass}
There is exactly $1$ $\GL_2(\OO)$-class of $8$-cone, $5$ $\GL_2(\OO)$-classes of $7$-cones, $10$ classes of $6$-cones, $11$ classes of $5$-cones, $9$ classes of $4$-cones, $4$ classes of $3$-cones, and $2$ classes of $2$-cones.  
\end{thm}
Table~\ref{tab:cones} gives the $\GL_2(\OO)$ classes of \Vor cones along with the number of $\Pi$ vertices, the rank of the cone, and the stabilizer in $\GL_2(\OO)$ of the cone.  The \verb+MAGMA+ small group type of each stabilizer is given for the finite groups.  In particular, the first component is the order of the group. Let $U^*$ denote the subgroup of upper triangular matrices in $\GL_2(\OO)$ such that the top left entry is a torsion unit.

\begin{table}\label{tab:cones}
\caption{$\GL_2(\OO)$-conjugacy classes of \Vor cones.}
\[
\begin{array}{c c}
\begin{array}{|c|c|c|c|}
\hline
\text{type}&\text{$\#$ of $v$} &\text{cone rank} &\text{stabilizer}\\\hline
A & 24&8&\ip{600, 54}\\
\hline
B_{1} & 9&7&\ip{30, 4}\\
B_{2} & 9&7&\ip{30, 4}\\
B_{3} & 12&7&\ip{100, 14}\\
B_{4} & 12&7&\ip{600, 54}\\
B_{5} & 7&7&\ip{50, 5}\\
\hline
C_{1} & 6&6&\ip{30, 4}\\
C_{2} & 6&6&\ip{10, 2}\\
C_{3} & 7&6&\ip{10, 2}\\
C_{4} & 6&6&\ip{60, 11}\\
C_{5} & 6&6& \ip{50, 5}\\
C_{6} & 6&6&\ip{20, 2}\\
C_{7} & 7&6&\ip{10, 2}\\
C_{8} & 6&6&\ip{20, 2}\\
C_{9} & 6&6& \ip{30, 4}\\
C_{10} & 6&6&\ip{50, 5}\\
\hline
D_{1} & 5&5&\ip{10, 2}\\
D_{2} & 5&5&\ip{10, 2}\\
D_{3} & 5&5&\ip{20, 5}\\
D_{4} & 5&5&\ip{10, 2}\\
D_{5} & 5&5&\ip{10, 2}\\
D_{6} & 5&5&\ip{10, 2}\\
D_{7} & 5&5& \ip{10, 2}\\
D_{8} & 5&5& \ip{10, 2}\\
D_{9} & 5&5&\ip{20, 5}\\
D_{10} & 5&5&\ip{100, 14}\\
D_{11} & 6&5&\ip{20, 5}\\
\hline

\end{array}
& \begin{array}{|c|c|c|c|}
\hline
\text{type}&\text{$\#$ of $v$} &\text{cone rank} &\text{stabilizer}\\\hline
E_{1} & 4&4&\ip{20, 5}\\
E_{2} & 4&4&\ip{20, 2}\\
E_{3} & 4&4& \ip{20, 2}\\
E_{4} & 4&4&\ip{10, 2}\\
E_{5} & 4&4&\ip{10, 2}\\
E_{6} & 4&4&\ip{20, 5}\\
E_{7} & 4&4&\ip{20, 5}\\
E_{8} & 4&4&\ip{10, 2}\\
E_{9} & 4&4&\ip{200, 31}\\
\hline
F_{1} & 3&3&\ip{60, 11}\\
F_{2} & 3&3& \ip{20, 5}\\
F_{3} & 3&3&\ip{100, 16}\\
F_{4} & 3&3&\ip{20, 5}\\
\hline
G_{1} & 2&2&\ip{200, 31}\\
G_{2} & 2&2&\Gamma_\infty\\
\hline
\end{array}
\end{array}\]
\end{table}

\subsection{Classification of forms}
Note that Theorem~\ref{thm:coneclass} gives a classification of binary Hermitian forms over $F$ based on the configuration of the minimal vectors.  Indeed, by duality, the vertices of the cones that arise above correspond to $\GL_2(\OO)$-classes of configurations of minimal vectors for forms over $F$ \cite{AM}.  For example, there are 4 distinct classes of cones with 3 vertices.  Hence there are 4 distinct $\GL_2(\OO)$-types of binary Hermitian forms over $F$ with exactly 3 minimal vectors.  One can distinguish the types by studying the configuration of minimal vectors.  

Let $[\cdot]$ denote the subset of the minimal vectors of the perfect form in the order given in \eqref{eq:perfectmin}.  For example, $[5,23] = \{e_1,e_2\}$.  Since $F$ has class number one, every primitive vector in $\OO^2$ is $\GL_2(\OO)$-conjugate to $e_1$.  Combined with Theorem~\ref{thm:coneclass}, one computes the following.
\begin{thm}
Let $\phi$ be a binary Hermitian form over $F$.  Then $M(\phi)$ is $\GL_2(\OO)$ conjugate to exactly one of the following.
\begin{gather*}
[5],\ [5, 20],\ [ 5, 23 ],\ 
[ 5, 8, 23 ],\ 
[ 5, 22, 23 ],\ 
[ 5, 20, 23 ],\ 
[ 5, 10, 23 ],\ 
[ 5, 8, 22, 23 ],\ \\
[ 4, 5, 8, 23 ],\ 
[ 5, 8, 10, 23 ],\ 
[ 5, 8, 18, 23 ],\ 
[ 5, 18, 20, 23 ],\ 
[ 5, 15, 17, 23 ],\ 
[ 5, 19, 20, 23 ],\ \\ 
[ 5, 18, 19, 23 ],\  
[ 5, 20, 23, 24 ],\ 
[ 4, 5, 8, 18, 23 ],\ 
[ 5, 8, 10, 12, 23 ],\ 
[ 5, 8, 20, 22, 23 ],\ \\ 
[ 4, 5, 8, 9, 23 ],\ 
[ 5, 8, 18, 19, 23 ],\ 
[ 5, 8, 12, 20, 23 ],\ 
[ 5, 8, 18, 22, 23 ],\ 
[ 5, 18, 19, 20, 23 ],\ \\ 
[ 5, 20, 22-24 ],\ 
[ 5, 9, 15, 17, 23 ],\ 
[ 5, 8, 18, 19, 20, 22 ],\ 
[ 5, 8, 10, 12, 21, 23 ],\ \\ 
[ 5, 8, 12, 20, 21, 23 ],\ 
[ 5, 8, 18, 19, 20, 22, 23 ],\  
[ 5, 8, 20, 22-24 ],\ 
[ 5, 8, 9, 15, 17, 23 ],\  \\
[ 5, 8, 10, 22-24 ],\ 
[ 3-5, 8, 15, 18, 23 ],\ 
[ 4, 5, 8, 18, 22, 23 ],\ 
[ 4, 5, 8, 9, 15, 23 ],\ \\ 
[ 5, 9, 15-17, 23 ],\ 
[ 5, 8, 10, 12, 20-24 ],\ 
[ 3, 5, 8, 18-20, 22-24 ],\ \\ 
[ 3-5, 7, 8, 10, 13, 15, 18, 22-24 ],\ 
[ 1, 4, 5, 8, 9, 11, 13-17, 23 ],\ 
[ 5, 6, 9, 15-17, 23 ],\ \\
[1-24]
\end{gather*}
   
\end{thm}
\end{section}

\bibliography{../references}    
\end{document}